\documentclass[12pt]{amsart} 
\usepackage{amssymb}
\usepackage{latexsym}
\usepackage{amsmath}
\usepackage{amsthm}
\usepackage{amsfonts}

\newtheorem{theorem}{Theorem}[section]
\newtheorem{lemma}[theorem]{Lemma}
\newtheorem{corollary}[theorem]{Corollary}
\newtheorem{proposition}[theorem]{Proposition}
\theoremstyle{definition}

\newtheorem{remark}[theorem]{Remark}

\numberwithin{equation}{section}
\numberwithin{theorem}{section}
\allowdisplaybreaks

\DeclareFontFamily{U}{mathx}{}
\DeclareFontShape{U}{mathx}{m}{n}{<-> mathx10}{}
\DeclareSymbolFont{mathx}{U}{mathx}{m}{n}
\DeclareMathAccent{\widecheck}{0}{mathx}{"71}

\newcommand{\R}{\mathbb R}
\newcommand{\N}{\mathbb N}
\newcommand{\C}{\mathbb C}
\newcommand{\T}{\mathbb T}

\newcommand{\lexpa}{L^{\alpha}_{\textnormal{exp}}}
\newcommand{\tinf}{T_{\infty}}
\newcommand{\linf}{L^\infty}
\newcommand{\logl}{L\textnormal{log} L}
\newcommand{\lexp}{L_{\textnormal{exp}}}
\newcommand{\logla}{L(\textnormal{log } L)^{\alpha}}
\newcommand{\loglb}{L(\textnormal{log } L)^{\beta}}
\newcommand{\tlog}{T_{\textnormal{log}}}
\newcommand{\simb}{\text{sim }\mathcal{B}}
\newcommand{\mcB}{\mathcal B}

\title[The  finite   Hilbert transform on $\linf$]
{The finite Hilbert transform acting   on $\linf$}

\author[G. P.  Curbera]{Guillermo P. Curbera}
\address{Facultad de Matem\'aticas \& IMUS,
Universidad de Sevilla, 
Calle Tarfia s/n,  Sevilla 41012, Spain}
\email{curbera@us.es}

\author[S. Okada]{Susumu Okada}
\address{112 Marcorni Crescent, Kambah, ACT 2902, Australia}
\email{sus.okada@outlook.com}

\author[W.J. Ricker]{Werner J. Ricker}
\address{Math.--Geogr.\  Fakult\"at, Katholische Universit\"at
Eichst\"att--Ingolstadt, D--85072 Eichst\"att, Germany}
\email{werner.ricker@ku.de}

\makeatletter
\@namedef{subjclassname@2020}{%
  \textup{2020} Mathematics Subject Classification}
\makeatother

\subjclass[2020]{Primary 44A15, 46E30; Secondary  47A53, 47B34.}
\keywords{Finite Hilbert transform, airfoil equation, $L^\infty$, Zygmund space $\lexp$.}
\begin{document}

\begin{abstract} 
The action of the finite Hilbert transform  defined 
on $L^\infty(-1,1)$ and taking its values 
in the Zygmund space $\lexp(-1,1)$ is studied in detail. 
This is a reciprocal situation to the investigation recently 
undertaken in \cite{COR-asnsp} 
of the finite Hilbert transform defined on the Zygumd space $\logl(-1,1)$
and taking its values in $L^1(-1,1)$.
The fact that both $L^\infty(-1,1)$ and  $\lexp(-1,1)$ fail to  be separable
generates new features not present in \cite{COR-asnsp}.
\end{abstract}

\maketitle


\section{Introduction}
\label{S1}


The finite Hilbert transform $T(f)$ of $f\in L^1:=L^1(-1,1)$ is the  principal value integral
\begin{equation*}
T(f)(t):=\lim_{\varepsilon\to0^+} \frac{1}{\pi}
\left(\int_{-1}^{t-\varepsilon}+\int_{t+\varepsilon}^1\right) \frac{f(x)}{x-t}\,dx ,
\end{equation*}
which exists for a.e.\ $t\in(-1,1)$ and is a measurable function.
Its study was motivated by problems in Aerodynamics (see
\cite{cheng-rott},  \cite{reissner}, \cite{sohngen}, \cite{tricomi-1}, \cite{tricomi}) and 
in Elasticity Theory
(via the study of one-dimensional singular integral operators; 
see  \cite{duduchava}, \cite{gakhov}, \cite{gohberg-krupnik-1},
\cite{gohberg-krupnik-2},  \cite{mikhlin-prossdorf}, \cite{muskhelishvili}).
These applications arise via the  solution of the \textit{airfoil  equation} 
\begin{equation}\label{airfoil}
\frac{1}{\pi}\, \mathrm{ p.v.}  \int_{-1}^{1}\frac{f(x)}{x-t}\,dx=g(t),
\quad \mathrm{a.e. }\; t\in(-1,1),
\end{equation}
where $g$ is given and $f$ is to be found.

Due to M. Riesz's theorem, the finite Hilbert transform  operator 
maps $L^p:=L^p(-1,1) $ continuously into itself whenever $1 < p < \infty$.
By the 1990s the $L^p$ theory for the finite Hilbert transform (FHT, in short) was
well established; see \cite[Ch.11]{king}, \cite{OE} and \cite[\S4.3]{tricomi}.
In 1991, a celebrated paper  of Gel'fand and Graev, \cite{GG}, renewed an interest 
in the FHT by exposing its role in the inversion of the Radon transform and hence, its 
applications  to image reconstruction in Tomography;
see, for example, \cite{APS}, \cite{bertola-etal}, \cite{katsevich-tovbis}, \cite{sidky-etal}.

The $L^p$ theory for the FHT on $(-1,1)$ is rather different than for the  Hilbert transform on $\R$ 
or its periodic version on $\T$. Namely, Tricomi showed that $T\colon L^p\to L^p$
is a Fredholm operator for $1<p<2$ (with full range and one-dimensional kernel) and for
$2<p<\infty$ (injective and with one-codimensional kernel); 
see \cite[Example 4.21]{okada-ricker-sanchez}. 
A deeper understanding of the 
FHT is obtained when considering its action on rearrangement invariant (r.i., in short) 
spaces on $(-1,1)$; see \cite{COR-ampa},  \cite{COR-mh},  \cite{COR-am}, 
\cite{COR-mathnach}. For example, 
for the FHT operator 
$T_X\colon X\to X$  acting on a r.i.\ space $X$ with non-trivial Boyd indices, 
the following alternative  for its fine spectrum  holds:
 \begin{itemize}
\item[(a)] The point spectrum: 
$\sigma_{\mathrm{pt}}(T_X)\not=\emptyset
 \iff L^{2,\infty}\subseteq X.$
\item[(b)] The residual spectrum: 
$\sigma_{\mathrm{r}}(T_X)\not=\emptyset
 \iff X\subseteq L^{2,1}.$
 \item[(c)] The continuous spectrum: 
$\sigma_{\mathrm{c}}(T_X)=\sigma(T_X)
 \iff L^{2,\infty}\not\subseteq X \text{ and } X\not\subseteq L^{2,1},$
\end{itemize}
where $L^{2,1}$ and $L^{2,\infty}$ are the usual Lorentz 
spaces; \cite[Proposition 5.1]{COR-am}.

Note that while the FHT operator 
maps $L^p$ continuously into itself for $1 < p < \infty$ it 
does not map $\linf$ into itself neither $L^1$ into itself. In the latter case, 
$L^1$ is mapped continuously into $L^0(-1,1)$, due to   Kolmogorov's Theorem,
\cite[Theorem III.4.9(b)]{BS}. In \cite{COR-asnsp} we studied the situation when
$T\colon\logl\to L^1$, where $\logl$ is the  classical Zygmund space, 
which is a r.i.\ space  close to $L^1$ in the sense that it contains all $L^p$ spaces, for $p>1$.

The aim of this paper is to study the action of the FHT on $\linf$. In this case,
the operator $T\colon\linf\to\lexp$, denoted by $\tinf$, 
maps $\linf$ continuously into $\lexp$, where  the  classical Zygmund space 
$\lexp$ is a r.i.\ space  close to $\linf$ in the 
sense that it is contained in all $L^p$ spaces, for $p<\infty$. 
With the objective of solving the airfoil equation \eqref{airfoil}, 
we also need to investigate the auxiliary operator $\widecheck T$, defined by
\begin{equation*}
\widecheck{T}(f) :=-wT\Big(\frac{f}{w}\Big),
\end{equation*}
with $w(x):=\sqrt{1-x^2}$, for $x\in(-1,1)$, and to
show that  $\widecheck T\colon L^\infty\to\lexp$ boundedly. 
This, together with the characterization of the range $T(L^\infty)\subsetneqq\lexp$
of $\tinf$,  allows us to identify the inverse operator
$T^{-1}\colon T(L^\infty)\to L^\infty$ and hence, to solve the airfoil equation.
We end the paper by showing that
$\tinf$ is optimally defined, that is, $T\colon L^\infty\to\lexp$ does not
admit a continuous, linear, $\lexp$-valued extension to any 
strictly larger r.i.\ domain space within $L^1$.


\section{Preliminaries}
\label{S2}


The measure space considered in this paper is 
 the Lebesgue measure $\mu$ on $\R$ restricted to the Borel $\sigma$-algebra $\mcB$ 
of the open interval $(-1,1)$,
which we again denote by $\mu$. The vector space of all $\C$-valued, $\mcB$-simple functions 
is denoted by $\simb$.   Denote by $L^0:=L^0(\mu)$ the vector space of all 
$\C$-valued measurable functions on $(-1,1)$.  Measurable functions 
on $(-1,1)$ which  
coincide a.e.\  are identified.  With respect to the a.e.\  pointwise order
for its positive cone, $L^0$ 
is a complex vector lattice.  The space $L^0$ is also a metrizable topological
vector space for the topology of convergence in measure.  
An order ideal  $X$ of $L^0$ is called a \textit{Banach 
function space} (briefly, B.f.s.) over the measure space $((-1,1), \mcB, \mu)$ 
if $\simb\subseteq X$ (equivalently, $\linf \subseteq X$)
and if $X$ is equipped with a lattice norm $\|\cdot\|_X$ for which 
it is complete.  For brevity, we speak of  $X$ as a B.f.s.\ over $(-1,1)$.  
Since $\mu$ is a finite measure, necessarily $X\subseteq L^1$ 
(cf. Definitions 1.1 and 1.3 in \cite[Section II.1]{BS})
with a continuous inclusion
(as it is a positive linear map between Banach lattices).
A typical example of a B.f.s.\ over $ (-1,1)$ is the Lebesgue space $L^p:= L^p(\mu)$ for 
each $ 1 \le p \le \infty$.

A function  $f$ in a B.f.s\ $X$ over $(-1,1)$ has 
\textit{absolutely continuous norm} if $\|f\chi_{A_n}\| \downarrow 0$ 
whenever $A_n \downarrow\emptyset$ in $\mcB$.
The subspace 
$X_a$ of $X$ consisting of all functions which have
absolutely continuous norm is an order ideal
in $X$. That is, $X_a$ is a closed linear subspace with the property: 
$f\in X_a$ and $|g|\le|f|$ $\mu$-a.e.\
implies that $g\in X_a$. A B.f.s.\ $X$ 
has \textit{$\sigma$-order continuous} (or 
\textit{absolutely continuous}) norm if $X=X_a$, that is, 
for $f\in X$ given, $\|f_n\|_X\downarrow 0$ for every sequence 
$(f_n)_{n=1}^\infty\subseteq L^0$ satisfying $|f|\ge|f_n|\downarrow 0$ $\mu$-a.e.
For these facts and notions, see  \cite[Section I.3]{BS}.
If $\linf\subseteq X_a$, then $X_a$ is also a B.f.s.\ over $(-1,1)$. It is known that
$(\linf)_a=\{0\}$. A B.f.s.\  $X$ has the \textit{Fatou property} if, whenever $(f_n)_{n=1}^\infty$ 
is a norm-bounded, increasing sequence of non-negative functions in $X$ for which
$f=\sup_{n\in\N} f_n $ exists in $L^0$, then necessarily  $f \in X$ 
and $\|f_n\|_X\uparrow \|f\|_X$.
The \textit{associate space} $X'$ of $X$  is  the vector sublattice of $L^0$ 
defined by
$X':= \{g\in L^0: fg \in L^1 \text{  for every  } f\in X\}$
equipped with the lattice norm $\|g\|_{X'} : = \sup\{|\int_{-1}^1 fg\,d\mu|: 
f \in X \text{ with   } \|f\|_X\le 1\}$.  It turns out that $X'$ is a closed linear subspace of the dual 
Banach space $X^*$ of $X$.  Moreover, $X'$ is also a B.f.s.\ over $(-1,1)$
which always has the Fatou property.  
If $X$ has  $\sigma$-order  continuous norm, then $X'= X^*$; 
see, for example,  \cite[p.29]{LT}.

The decreasing rearrangement $f^*\colon [0,2]\to [0,\infty]$  
 of a function $f \in L^0$  is the
right continuous inverse of its distribution function:
$ \lambda\mapsto \mu\big(\{t\in (-1,1):\,|f(t)|>\lambda\}\big)$ for $\lambda\ge 0$.
A \textit{rearrangement invariant}  space $X$ over $(-1,1)$ is a
B.f.s.\ with the Fatou property
such that, if $g^*\le f^*$ with $f\in X$,
then $g\in X$ and $\|g\|_X\le\|f\|_X$. In this case, its associate space $X'$ 
is also a r.i.\  space, \cite[Proposition II.4.2]{BS}.


For $\alpha\ge0$, the Zygmund space $\lexpa:=\lexpa(-1,1)$ 
consists of all measurable functions $f$
on $(-1,1)$  for which there exists a constant
$\lambda=\lambda(f)>0$ such that
$$
\int_{-1}^1\exp\big(\lambda|f(x)|\big)^{1/\alpha}\,dx<\infty,
$$
\cite[Definition IV.6.11]{BS}. 
For $\alpha=0$ the space $\lexp^0$ is interpreted to be $\linf$. For each
$\alpha\ge0$ the space $\lexpa$ is a r.i.\ space with respect to the  norm 
$$
\|f\|_{\lexpa}:=\sup_{0<t<2}\frac{\int_0^tf^*(s)\,ds}{t(\log\big(\frac{2e}{t}\big))^\alpha},
\quad f\in\lexpa,
$$
\cite[Lemma IV.6.12]{BS}. 
An equivalent norm to the above norm in $\lexpa$ is given by
$$
\sup_{0<t<2}\frac{f^*(t)}{\big(\log\big(\frac{2e}{t}\big)\big)^\alpha},
\quad f\in\lexpa;
$$
see \cite[Theorem 5.3 in \S II.5.2, p.115]{KPS}. The spaces $\lexpa$ are non-separable 
and hence, their norm is not $\sigma$-order  continuous. They are
close to $\linf$ in the sense that
$\linf\subseteq \lexpa\subseteq L^p$ for all 
$\alpha\ge0$ and all $1\le p<\infty$;  for the case $\alpha=1$,
see  \cite[Theorem IV.6.5]{BS}. 
Observe that $\lexp^\alpha\subseteq \lexp^\beta$ for all $0\le\alpha\le\beta$.

For $\alpha=1$,  the Zygmund  space $\lexp:=L_{\text{exp}}^1$
consists of all functions having exponential integrability, \cite[Definition IV.6.1 and Lemma IV.6.2]{BS},
equivalently, functions $f$ on $(-1,1)$ for which there exists a constant $c=c(f)>0$ such that
$$
f^*(t)\le c\log\Big(\frac{2e}{t}\Big),\quad 0<t<2.
$$
The r.i.\ space $\lexp$ is a non-separable Marcinkiewicz space and the closure of $\simb$
 in $\lexp$, denoted by $(\lexp)_b$, is the B.f.s.\ 
consisting of all measurable functions $f$ on $(-1,1)$ satisfying
\begin{equation}\label{eq-2.1}
\lim_{t\to0^+}\frac{\int_0^tf^*(s)\,ds}{t\log\big(\frac{2e}{t}\big)}=0,
\end{equation}
\cite[Lemma 5.4 in \S II.5.2, p.116]{KPS}; it coincides with the separable space 
$(\lexp)_a$ and does 
not have the Fatou property.
The r.i.\ space $\lexp$ is also an Orlicz space, \cite[Section IV.8]{BS}. Define
a Young function $\Phi$ on $[0,\infty)$ by $\Phi(t):=e^t-t-1$, which satisfies
$$
\int_{-1}^1\Phi(|f|)\,d\mu\le \int_{-1}^1 e^{|f|}\,d\mu\le
3+2\int_{-1}^1\Phi(|f|)\,d\mu,\quad f\in L^0.
$$
The previous inequality implies that the corresponding Orlicz space 
coincides with $\lexp$ (with equivalent norms) as $\Phi$ is equivalent to the Young
function given in Example IV.8.3(d) in \cite{BS}.
Accordingly, 
\begin{equation}\label{new2.2}
(\lexp)_a=(\lexp)_b=\Big\{f\in\lexp: \int_{-1}^1 e^{\lambda|f|}\,d\mu<\infty
\text{ for \textit{every} } \lambda\ge0\Big\};
\end{equation}
see \cite[p.120]{LT} (and \cite[Ch. 2, Sect. 3, Theorem 1]{Lu} for a proof).
It follows that $L^\infty\subsetneqq (\lexp)_a\subsetneqq \lexp$. Indeed, direct 
calculations via \eqref{new2.2} show that 
$f(x):=|\log (x)|\chi_{[0,1)}(x)$ and $g(x):=\chi_{[0,1)}(x)\log(1+|\log(x)|)$,
for $x\in(-1,1)$, belong to        
$\lexp \setminus (\lexp)_a$  and    
$(\lexp)_a\setminus L^\infty$, respectively.

For $\alpha\ge0$,  the Zygmund space $\logla$  
consists of all measurable functions $f$
on $(-1,1)$ for which either one of the following two equivalent conditions hold:
$$
\int_{-1}^1|f(x)|(\log(2+|f(x)|)^\alpha\,dx<\infty,\quad \int_0^2 f^*(t)\Big(\log
\Big(\frac{2e}{t}\Big)\Big)^\alpha\,dt<\infty,
$$
\cite[Definition IV.6.11, Lemma IV.6.12]{BS}. 
The space $\logla$ is  r.i.\  with $\sigma$-order continuous norm  
given by
$$
\|f\|_{\logla}:=\int_0^2 f^*(t)\Big(\log\Big(\frac{2e}{t}\Big)\Big)^\alpha\,dt,\quad f\in\logla.
$$
For each $\alpha\ge0$,  the associate space of $\logla$ is 
$\lexpa$, \cite[Theorem 5.2 in \S II.5.2, p.112]{KPS}.
Since $\logla$ has $\sigma$-order continuous norm, 
$(\logla)^*=(\logla)'=\lexpa$. The spaces $\logla$ are
close to $L^1$ in the sense that
$L^p\subseteq\logla\subseteq L^1$ for all $\alpha\ge0$ and all $1<p<\infty$; 
for the case $\alpha=1$, see \cite[Theorem IV.6.5]{BS}. 
Moreover, $\loglb\subseteq \logla$
for all $0\le\alpha\le\beta$.

For $\alpha=1$  the Zygmund  space  $L(\textnormal{log} L)^1$ is, of course, $\logl$.
The case when $\alpha=0$ is interpreted as $L(\textnormal{log} L)^0=L^1$.


The \textit{lower} and \textit{upper Boyd indices} of  a r.i.\  space
$X$ are two indices, $\underline{\alpha}_X$ and  $\overline{\alpha}_X$, which 
measure the effect of dilations in $X$ and satisfy 
$0\le\underline{\alpha}_X\le \overline{\alpha}_X\le1$,
\cite[Definition III.5.12]{BS}. 
The boundedness of the finite Hilbert  transform
$T\colon X\to X$ is equivalent to $X$ having non-trivial  Boyd indices, that is, 
$0<\underline{\alpha}_X\le \overline{\alpha}_X<1$; 
see  \cite[\S II.8.6, pp.170--171]{KPS}.


For additional details concerning the Zygmund spaces $\logla$
and $\lexpa$ we refer to \cite[Ch.III]{BR}, for example.


\section{The finite Hilbert transform on $\linf$}
\label{S3}


In this section we establish some basic properties of the finite Hilbert transform
when it acts on $L^\infty$ and takes its values in $\lexp$.

The following result, a combination of Theorem 3.1 and
Corollary 3.3 in  \cite{COR-asnsp}, will be repeatedly used in the sequel.

\begin{proposition}\label{p-3.1}
Let  $f\in L^1$ and $g\in \logl$ satisfy   $fT(g\chi_A)\in L^1$, 
for every  set $A\in\mathcal B$.
Then  $gT(f)\in L^1$ and the following Parseval formula is valid:
\begin{equation}\label{eq-3.1}
\int_{-1}^{1}fT(g)=-\int_{-1}^{1}gT(f).
\end{equation}

In particular, if $f\in L^\infty$ and $g\in\logl$, then  the above 
assumptions are satisfied and hence, \eqref{eq-3.1} is valid.
\end{proposition}


\begin{proposition}\label{p-3.2}
For $\alpha\ge0$, the finite Hilbert transform 
$T\colon \lexpa\to L^{\alpha+1}_{\textnormal{exp}}$ is a bounded operator.
In particular, $T\colon L^\infty\to \lexp$ is a bounded operator.
\end{proposition}

\begin{proof}
It is known  that $T\colon L(\textnormal{log } L)^{\alpha+1}\to L(\textnormal{log } L)^{\alpha}$ 
is a bounded operator, \cite[Proposition 4.4]{COR-asnsp}.
Accordingly,  its adjoint operator 
$T^*\colon \lexpa\to L^{\alpha+1}_{\textnormal{exp}}$ boundedly.

Given $f\in \lexpa\subseteq L^1$ and 
$g\in L(\textnormal{log } L)^{\alpha+1}\subseteq\logl$, it follows from 
Proposition \ref{p-3.1} that $gT(f)\in L^1$ and 
$$
\int_{-1}^{1}gT(f)=-\int_{-1}^{1}fT(g)=-\int_{-1}^{1}gT^*(f).
$$
Thus,  $T(f)\in L^{\alpha+1}_{\textnormal{exp}}$. It follows that  $T^*=-T$. Consequently, 
$T\colon \lexpa\to L^{\alpha+1}_{\textnormal{exp}}$ boundedly.

For $\alpha=0$, we have $L_{\text{exp}}^0=\linf$ and 
$L_{\text{exp}}^1=\lexp$ and so 
$T\colon \linf\to \lexp$ boundedly.
\end{proof}

When convenient, the operator $T\colon L^\infty\to \lexp$ will
also be denoted by $\tinf$, where the subscript  indicates briefly that the domain 
space is $\linf$.


\begin{remark}\label{r-3.3}
It is shown
in  Proposition 4.15 of \cite{COR-asnsp} that $T\colon\logl\to L^1$ is not  compact.
According to Schauder's Theorem  also its adjoint 
operator $T\colon \linf \to \lexp$ (cf.\ Proof of Proposition 3.2) is not  compact. 
A similar argument, via Gantmacher's Theorem, shows that 
$T\colon \linf \to \lexp$ is not a weakly compact operator
(since this is the case for $T\colon\logl\to L^1$, 
\cite[Theorem 4.7]{COR-mathnach}).

Moreover, the operator $T\colon \linf \to \lexp$ is not order bounded. For, if so,
then there exists $0\le f\in\lexp$ such that $|T(\chi_A)|\le f$, for all $A\in\mathcal{B}$. 
Since $\lexp\subseteq L^2$,
say, also $f\in L^2$, and so $\{T(\chi_A):A\in\mathcal{B}\}$ is an order bounded subset of $L^2$.
This is a contradiction to Proposition 3.2(vii) in \cite{COR-qm},
with  $X=L^2$ there.
\end{remark}


A central role is played by the bounded, continuous function 
\begin{equation}\label{eq-3.2}
 w(x) : = \sqrt{1-x^2}, \quad x \in (-1,1).
\end{equation}

\begin{proposition}\label{p-3.4}
The finite Hilbert transform $T\colon L^\infty\to \lexp$ is  injective.
\end{proposition}

\begin{proof}
Theorem 3.4 in \cite{COR-asnsp} implies, 
for $f\in\logl$, that $T(f)=0$ if and only if $f=c/w$ for some 
$c\in\C$. Since $1/w\notin \linf$, it follows that $T\colon L^\infty\to \lexp$ is injective.
\end{proof}


\begin{proposition}\label{p-3.5}
The  smallest r.i.\ space over $(-1,1)$ which contains $T(\linf)$ is $\lexp$.
\end{proposition}

\begin{proof}   
Let $X$ be a r.i.\ space over $(-1,1)$ such that $T(\linf)\subseteq X$.
Then $T(\chi_{(-1,1)})\in X$. From the definition of $T$ it follows  
that $T(\chi_{(-1,1)})(x)=\frac1\pi\log\big(\frac{1-x}{1+x}\big)$, for $x\in(-1,1)$. 
Direct calculation also shows that 
$$
\left(\log\Big(\frac{1-x}{1+x}\Big)\right)^*(t)=\log\Big(\frac{4-t}{t}\Big),\quad 0<t<2.
$$
Consider the function $F(x):=1+|\log\big(\frac{1-x}{1+x}\big)|$, for $x\in(-1,1)$,  
which belongs to $X$ (as the constant function $1\in L^\infty$  belongs to $X$). Then
$$
F^*(t):=1+\left(\log\Big(\frac{1-x}{1+x}\Big)\right)^*(t)\ge \log(2e/t) ,\quad 0<t<2.
$$
Let $f\in\lexp$. Then there exists a  positive constant $M_f$ such that
$f^*(t)\le M_f \log(2e/t)$, for $0<t<2$. It follows that 
$f^*(t)\le  M_fF^*(t)$, for $0<t<2$. Since $X$ is r.i.\ and $F\in X$, this
implies that $f\in X$, \cite[Theorem II.4.6]{BS}. So,  $\lexp\subseteq X$.
\end{proof}

\begin{remark}\label{r-3.6} 
In the proof of Proposition \ref{p-3.5} it was noted that 
$h:=\pi T(\chi_{(-1,1)})$ satisfies  $h^*(s)\ge   (\log(2e/s)-1)$ for
$s\in(-1,1)$, that is 
$$
\frac{\int_0^th^*(s)\,ds}{t\log(2e/t)}\ge
\frac{\int_0^t(\log(2e/s)-1)\,ds}{t\log(2e/t)},\quad 0<t<2.
$$
Direct calculation shows that 
$\lim_{t\to0^+}\frac{\int_0^t(\log(2e/s)-1)\,ds}{t\log(2e/t)}=1\not=0$
and so $h\not\in(\lexp)_a$; see \eqref{eq-2.1}.
Accordingly, $T(L^\infty)\not\subseteq(\lexp)_a$.
Since $(\lexp)_a$ does not have the Fatou property,
it is not a r.i.\ space and so the previous observation does
not follow from Proposition \ref{p-3.5}.
\end{remark}


\section{The auxiliary operator  $\widecheck{T}$ on $\linf$}
\label{S4}


Let $f\in L^1$ satisfy  $f/w\in L^1$. Kolmogorov's Theorem ensures that
$T(f/w)\in L^0$ and so  the function $\widecheck T(f)\in L^0$, 
given pointwise a.e.\ in $(-1,1)$ by
\begin{equation}\label{eq-4.1}
\widecheck{T}(f) :=-w\,T\,\Big(\frac{f}{w}\Big),
\end{equation}
is well defined. For $f\in L^1$, the function $\widehat T(f)\in L^0$, 
given pointwise a.e.\ in $(-1,1)$ by
\begin{equation}\label{eq-4.2}
\widehat{T}(f) :=-\frac{1}{w}T(wf),
\end{equation}
is also well defined.

The operator $f\mapsto \widehat T(f)$ arose when solving
the airfoil equation \eqref{airfoil} (i.e., for the inversion of $T$) for 
$T\colon L^p\to L^p$ with $1<p<2$, and, more generally,
for $T\colon X\to X$ with $X$ a r.i.\ space
satisfying $1/2<\underline{\alpha}_X\le \overline{\alpha}_X<1$.
The operator $f\mapsto \widecheck T(f)$ arose when determining the
solution of the airfoil equation for 
$T\colon L^p\to L^p$ with $2<p<\infty$ and for 
$T\colon X\to X$ with $X$ a r.i.\ space
satisfying $0<\underline{\alpha}_X\le \overline{\alpha}_X<1/2$.
We refer to \cite{COR-ampa} and the references therein.


In this section we investigate the operators $\widecheck T$
and $\widehat T$ acting between (possibly) different pairs of
r.i.\ spaces $X$ and $Y$, with no restriction on their Boyd indices.

It is well known that positive operators between B.f.s.' 
are (automatically) continuous. The finite Hilbert transform is far from being positive.
However, a related result is available.

\begin{lemma}\label{l-4.1}
Let $X, Y$ be B.f.s.' over $(-1,1)$.
\begin{itemize}
\item[(i)] If $T(X)\subseteq Y$, then $T\colon X\to Y$ is continuous.
\item[(ii)] If $\widehat T(X)\subseteq Y$, then $\widehat T\colon X\to Y$ is continuous.
\item[(iii)] If $\widecheck T(X)\subseteq Y$, then $\widecheck T\colon X\to Y$ is continuous.
\end{itemize}
\end{lemma}

\begin{proof} 
(i) Let $f_n\to0$ in $X$ and $T(f_n)\to g$ in $Y$. Since 
$X\subseteq L^1$ continuously, also $f_n\to0$ in $L^1$.
Kolmogorov's Theorem implies that $T(f_{n_j})\to 0$ a.e.\ for some subsequence
$\{f_{n_j }\}_{j=1}^\infty$ of $\{f_{n}\}_{n=1}^\infty$. Also $T(f_{n_j})\to g$ in $L^1$ 
(as $Y\subseteq L^1$ continuously) and so
$T(f_{n_{j_k}})\to g$ a.e.\ for some subsequence
$\{f_{n_{j_k} }\}_{k=1}^\infty$ of $\{f_{n_j }\}_{j=1}^\infty$.
Since $T(f_{n_{j_k}})\to 0$ a.e., we can conclude that $g=0$. By the Closed Graph Theorem,
$T\colon X\to Y$ is continuous.

(ii) Let $f_n\to0$ in $X$ and $\widehat T(f_n)\to g$ in $Y$. 
Since $|wf_n|\le |f_n|$ for $n\in\N$, also $(wf_n)\to0$ in $X$. Hence,
$(wf_n)\to0$ in $L^1$.
Kolmogorov's Theorem implies that $T(wf_{n_j})\to 0$ a.e.\ for some subsequence
$\{f_{n_j }\}_{j=1}^\infty$ of $\{f_{n}\}_{n=1}^\infty$. Then also
$(-1/w)T(wf_{n_j})=\widehat T(f_{n_j})\to 0$ a.e. 
Moreover, $\widehat T(f_{n_j})\to g$ in $L^1$ and so 
$\widehat T(f_{n_{j_k}})\to g$ a.e.\ for some subsequence
$\{f_{n_{j_k} }\}_{k=1}^\infty$ of $\{f_{n_j }\}_{j=1}^\infty$.
Since $(-1/w)T(wf_{n_{j_k}})=\widehat T(f_{n_{j_k} })\to 0$ a.e., 
we can conclude that $g=0$. By the Closed Graph Theorem,
$\widehat T\colon X\to Y$ is continuous.

(iii) Let $f_n\to0$ in $X$ and $\widecheck T(f_n)\to g$ in $Y$. 
We point out that $\widecheck T$ defined on $X$ means that $(f/w)\in L^1$ for every $f\in X$
and that $\widecheck T(f)\in Y$. Accordingly,
the positive linear operator $f\in X\mapsto (f/w)\in L^1$ is well defined and hence, 
it is continuous. So, $(f_n/w)\to0$ in $L^1$.
Kolmogorov's Theorem implies that $T(f_{n_j}/w)\to 0$ a.e.\ for some subsequence
$\{f_{n_j }\}_{j=1}^\infty$ of $\{f_{n}\}_{n=1}^\infty$.  Then also
$-wT(f_{n_j}/w)=\widecheck T(f_{n_j})\to 0$ a.e. 
Moreover, $\widecheck T(f_{n_j})\to g$ in $L^1$ and so
$\widecheck T(f_{n_{j_k}})\to g$ a.e.\ for some subsequence
$\{f_{n_{j_k} }\}_{k=1}^\infty$ of $\{f_{n_j }\}_{j=1}^\infty$.
Since $-wT(f_{n_{j_k}}/w)=\widecheck T(f_{n_{j_k} })\to 0$ a.e., 
we can conclude that $g=0$. By the Closed Graph Theorem,
$\widecheck T\colon X\to Y$ is continuous.
\end{proof}


The following result makes the relation between  the operators 
$\widecheck{T}$ and $\widehat{T}$  precise.

\begin{theorem}\label{t-4.2}
Let $X, Y$ be  r.i.\ spaces such that 
$X\subseteq\logl$ and $L^{2,\infty}\subseteq Y$. Suppose that 
$\widehat{T}\colon X\to Y$ boundedly. Then 
$$
\widecheck{T}\colon Y'\to X'
$$  boundedly.  Moreover,
for every $f\in Y'$ and $g\in X$ it is the case that 
\begin{equation}\label{eq-4.3}
\int_{-1}^1\widehat{T}(g)f=-\int_{-1}^1g\widecheck{T}(f).
\end{equation}
If, in addition, both $X$ and $Y$ are separable, then 
$(\widehat{T})^*=-\widecheck{T}$.
\end{theorem}

\begin{proof}
Let $f\in Y'$. Observe that 
$f/w\in L^1$ because 
$f\in Y'\subseteq ( L^{2,\infty})'=L^{2,1}$ and $1/w\in L^{2,\infty}$.
Hence, $\widecheck T(f)=-wT(f/w)\in L^0$.
Let $g\in X$. Given $A\in\mathcal{B}$, 
note that $\widehat T(g\chi_A)\in Y$, and so H\"older's inequality yields
$$
\frac{f}{w}T\big(wg\chi_A\big)= - f \widehat T\big(g\chi_A\big)\in L^1.
$$
This allows us to apply Proposition \ref{p-3.1} 
(with $(f/w)\in L^1$ in place of $f$ and $(gw)\in\logl$ in place of $g$)
to conclude that $-g\widecheck T(f)=(gw)\,T(f/w)\in L^1$ and
$$
\int_{-1}^1 g\widecheck{T}(f)=-\int_{-1}^1 gw\, T\Big(\frac{f}{w}\Big)
=\int_{-1}^1 \frac{f}{w} T(gw)=-\int_{-1}^1 f \widehat T(g).
$$
This verifies \eqref{eq-4.3} and, because $f\widehat{T}(g)\in L^1$
for every $g\in X$, also shows that $\widecheck T(f)\in X'$.
Hence, $\widecheck T(Y')\subseteq X'$ and so Lemma \ref{l-4.1}
implies that $\widecheck{T}\colon Y'\to X'$ is bounded.

When $X$ and $Y$ are both separable
(equivalently, have $\sigma$-order continuous norms, 
\cite[Corollary I.5.6]{BS}), we have $X^*=X'$ and $Y^*=Y'$. The
identity $(\widehat T)^*=-\widecheck T$ is then clear from \eqref{eq-4.3}.  
\end{proof}


\begin{corollary}\label{c-4.3}
For each $\alpha\ge0$ the operator 
$\widecheck{T}\colon \lexpa\to L^{\alpha+1}_{\textnormal{exp}}$ is bounded.
In particular, $\widecheck{T}\colon \linf\to\lexp$ is bounded.
\end{corollary}

\begin{proof}
We can apply Theorem \ref{t-4.2} to $X=L(\textnormal{log } L)^{\alpha+1}$ 
and $Y=\logla$ since 
$\widehat{T}\colon L(\textnormal{log } L)^{\alpha+1}\to \logla$ boundedly,
\cite[Proposition 4.5]{COR-asnsp}, with $X\subseteq \logl$ and 
$L^{2,\infty}\subseteq Y$.
\end{proof}

\begin{remark}\label{r-4.4}
(i) Since $\widehat{T}_p:=\widehat{T}\colon L^p\to L^p$ boundedly, for $1<p<2$, 
 \cite[Lemma 2.3]{OE}, it follows from
Theorem \ref{t-4.2} that 
$\widecheck{T}_{p'}:=\widecheck{T}\colon L^{p'}\to L^{p'}$ 
satisfies $(\widehat{T}_p)^*=-\widecheck{T}_{p'}$ with $2<p'<\infty$;
compare  this with  \cite[Proposition 2.6]{OE}.

(ii) For a  r.i.\ space $X$ with
$1/2<\underline{\alpha}_X\le \overline{\alpha}_X<1$, it is known that  
$\widehat{T}\colon X\to X$ boundedly, \cite[Theorem 3.2]{COR-ampa}.
If $X$ is also separable, then 
Theorem \ref{t-4.2} implies that  $\widecheck{T}\colon X'\to X'$ 
boundedly (denoted by $\widecheck{T}_{X'}$)
and $(\widehat{T}_X)^*=-\widecheck{T}_{X'}$
with $0<\underline{\alpha}_{X'}\le \overline{\alpha}_{X'}<1/2$;
compare with \cite[Theorem 3.3]{COR-ampa}.

(iii) A proof of Corollary \ref{c-4.3} avoiding the use of duality is also possible. 
This requires 
showing, for some constant $C>0$ and some $2<p_0<\infty$, that 
$\|\widecheck{T}\|_{L^p\to L^p}\le C p$ for all $p_0<p<\infty$, and then applying 
the classical result of Zygmund on extrapolation;
see \cite[Theorem XII.(4$\cdot$41)]{Z}, on p. 119 of volume II.
Obtaining the uniform estimate  needed for $\|\widecheck{T}\|_{L^p\to L^p}$ 
follows, via analogy,  the steps established in the corresponding study 
of the operator $\widehat T$, in particular,
by adapting Lemmas 4.6, 4.7 and 6.1 in \cite{COR-asnsp}.
\end{remark}


\section{Inversion of the FHT on $\linf$}
\label{S5}


In this section we study the action of  $T$ and $\widecheck{T}$ on $\linf$.
In particular, the range $T(L^\infty)$ of $T$ in $\lexp$ is 
characterized, which is then used to solve the airfoil equation.

Compare the following result with  \cite[Proposition 2.6]{OE} and 
\cite[Theorem 3.3]{COR-ampa}.

\begin{theorem}\label{t-5.1}  
The following assertions hold.
\begin{itemize}
\item[(i)]   The operator $\tinf\colon\linf\to\lexp$ is  injective.
\item[(ii)] Let $f\in\linf$. Then  $T(f)\in\lexp$ and 
$\widecheck{T} T(f)=f$.
\item[(iii)] The operator $Q\colon\linf\to\linf$  given by
\begin{equation}\label{eq-5.1}
Q (f):= \left(\frac1\pi\int_{-1}^1\frac{f}{w}\right)\chi_{(-1,1)},  
\quad  f\in {\linf} ,
\end{equation}
is a bounded, rank-$1$, projection of $L^\infty$ onto 
$\textnormal{span}(\chi_{(-1,1)})\subseteq L^\infty$ satisfying 
$\|Q\|_{\linf\to\linf}=1$.
Furthermore, 
$$
T \widecheck{T}(f)=f-Q(f),\quad f\in\linf. 
$$
\item[(iv)]  For each $f\in\linf$ it is the case that 
$$ 
\int_{-1}^1\frac{T(f)}{w}=0.
$$
\end{itemize}
\end{theorem}

\begin{proof}
Proposition \ref{p-3.4} establishes (i).
Let $f\in \linf$ and fix  any $2<p<\infty$, in which case $f\in L^p$. It follows from
Theorem 3.3 in  \cite{COR-ampa} for $X=L^p$ that (ii), (iii) and (iv) hold.
That $\|Q\|_{\linf\to\linf}=1$ follows from \eqref{eq-5.1} and $\int_{-1}^1(1/w)=\pi$.
\end{proof}

\begin{remark}\label{r-5.2} 
The operator $Q_{\lexp}\colon\lexp\to\lexp$ given by
$$
Q_{\lexp}(g):=\left(\frac1\pi\int_{-1}^1\frac{g}{w}\right)\chi_{(-1,1)},\quad g\in\lexp,
$$
is also a bounded, rank-1 projection of $\lexp$ onto 
$\textnormal{span}(\chi_{(-1,1)})\subseteq \lexp$. Indeed,
since $1/w\in\logl$, H\"older's inequality yields that $g/w\in L^1$ and
$$
\big|Q_{\lexp}(g)\big|
\le \frac1\pi\Big\|\frac1w\Big\|_{\logl}\|g\|_{\lexp}\chi_{(-1,1)},\quad g\in\lexp,
$$
with $\chi_{(-1,1)}\in L^\infty\subseteq\lexp$. 
The previous pointwise inequality on $(-1,1)$ implies
that $Q_{\lexp}(g)\in L^\infty\subseteq\lexp$ and that
$Q_{\lexp}$ is bounded with operator norm 
$$
\big\|Q_{\lexp}\big\|\le \frac1\pi\Big\|\frac1w\Big\|_{\logl}\|\chi_{(-1,1)}\|_{\lexp}
=\frac1\pi\Big\|\frac1w\Big\|_{\logl}.
$$
Clearly the range $Q_{\lexp}(\lexp)=\textnormal{span}(\chi_{(-1,1)})\subseteq L^\infty
\subseteq\lexp$.

Suppose now that a function  $g_0\in\lexp$ satisfies 
$\widecheck T(g_0)\in L^\infty$, in which case
$T\widecheck T(g_0)\in \lexp$.  To prove  that the identity
\begin{equation}\label{eq-5.1-2A}
T{\widecheck T}(g_0)=g_0-Q_{\lexp}(g_0)
\end{equation}
 holds in $\lexp$,  
 fix any $2<p<\infty$, in which case  $\lexp\subseteq L^p$. Moreover, the
identity
\begin{equation}\label{eq-5.1-2}
T\widecheck T(g_0)=g_0-Q(g_0)
\end{equation}
holds in $L^p$, \cite[Proposition 2.6]{OE}, where
$Q(f):=\big(\frac1\pi\int_{-1}^1\frac{f}{w}\big)\chi_{(-1,1)} \in L^p$ for $f\in L^p$. But, both
sides of the identity \eqref{eq-5.1-2} belong to $\lexp$ and the proof of \eqref{eq-5.1-2A} is complete.
\end{remark}


Recall that the function $1/w\in\logl=(\lexp)'$.
So, it defines a continuous linear functional $\varphi_{1/w}\colon\lexp\to\C$ by
$$
\varphi_{1/w}(g):=\Big\langle g,\frac{1}{w}\Big\rangle=
\int_{-1}^1\frac{g}{w},\quad g\in\lexp,
$$
which satisfies $\|\varphi_{1/w}\|_{\lexp^*}\le \|\frac1w\|_{\logl}$. Observe that
$$
\textnormal{Ker}(\varphi_{1/w})=\Big\{g\in\lexp: \int_{-1}^1(g/w)=0\Big\}
= \textnormal{Ker}(Q_{\lexp})
$$
is a proper, closed, infinite-dimensional subspace of $\lexp$.
Indeed, if $g\in\lexp$ is any odd function, then 
$\int_{-1}^1(g/w)=0$ and so $g\in \textnormal{Ker}(\varphi_{1/w})$.

Theorem \ref{t-5.1}(iv) implies that $T(\linf)\subseteq \textnormal{Ker}(\varphi_{1/w})$.
In particular, $\tinf$ is not surjective.


An important consequence of Theorem \ref{t-5.1} is the following
characterization of the range space of $\tinf$.

\begin{corollary}\label{c-5.2}  
A function $g\in \lexp$ belongs to the range space
$T(\linf)\subseteq\lexp$ if and only if it satisfies  both $\widecheck T(g)\in \linf$ and
$\varphi_{1/w}(g)=0$. That is, 
$$
T(\linf)=\Big\{g\in\textnormal{Ker}(\varphi_{1/w}):  \widecheck T(g)\in \linf\Big\}
=\Big\{g\in\textnormal{Ker}(Q_{\lexp}) :  \widecheck T(g)\in \linf\Big\}.
$$
\end{corollary}

\begin{proof}
Let $g\in \lexp$ satisfy $\widecheck T(g)\in \linf$ and 
$\int_{-1}^1(g/w)=0$, that is, $Q_{\lexp}(g)=0$. Remark 
\ref{r-5.2} implies that
$T(\widecheck T(g))=g$. Accordingly, $g\in T(\linf)$.

Conversely, let $g\in T(\linf)$. Then there exists $f\in \linf$ such that $g=T(f)\in\lexp$. 
From Theorem \ref{t-5.1}(ii) we have
$\widecheck T(g)=\widecheck TT(f)=f\in\linf$.
Moreover, Theorem \ref{t-5.1}(iv) implies that
$$
\int_{-1}^1\frac{g}{w}=\int_{-1}^1\frac{T(f)}{w}=0.
$$
\end{proof}


\begin{remark}\label{r-5.3}
We have seen that $\tinf(\linf)\subseteq \textnormal{Ker}(\varphi_{1/w})$. 
This inclusion is actually \textit{proper}. Indeed, let $X:=\logl$,
$Y:=L^1$ and $S\colon X\to Y$ denote the finite Hilbert
transform $T\colon\logl\to L^1$, briefly $\tlog$. Proposition
4.12(i) of \cite{COR-asnsp} implies that $S$ is \textit{not} surjective.
Moreover, $S^*=-\tinf$ is injective. According to
(a)$\iff$(c) in Theorem 4.15 of \cite{R}, with $S$ in place of $T$ there,
it follows that $\tinf(\linf)$ is \textit{not} norm-closed in $\lexp$.
Then $\tinf(\linf)$ cannot be norm-closed in the closed subspace
$\textnormal{Ker}(\varphi_{1/w})$ either. 
\end{remark}


We seek   necessary conditions on a function  $g$ to satisfy the condition 
$\widecheck{T}(g)\in\linf$ appearing in the characterization of the range space
$T(\linf)$ in Corollary \ref{c-5.2}. In this regard, the following result 
(cf.\ Theorem 2.3 in \cite{COR-mathnach})
is relevant. A further class of functions $g$ is given in Corollary \ref{c-5.7}.

\begin{proposition}\label{p-5.4}
Let $g$ satisfy a uniform $\lambda$-H\"older condition on $(-1,1)$ for 
$\lambda\in(0,1]$, that is, for some constant $K_g>0$ we have
$$
|g(s)-g(t)|\le K_g|s-t|^{\lambda},\quad s,t\in(-1,1).
$$
Then, with $B(\cdot,\cdot)$ denoting the Beta function, it is the case that
$$
w(t)\Big|\Big(T\Big(\frac{g}{w}\Big)\Big)(t)\Big|\le \frac{2}{\pi} K_g\, B(1/2,\lambda),
\quad t\in(-1,1).
$$
In particular, $\widecheck{T}(g)\in \linf$.
\end{proposition}


The range space $T(\linf)$ has been identified in  Corollary \ref{c-5.2}. 
The next result presents various facts aimed 
at a better understanding  of  $T(\linf)$.  
For the definition of the weak-$\ast$ topology $\sigma(\lexp,\logl)$ on
$\lexp=(\logl)^*$ see \cite[Section 3.14]{R}, for example.

\begin{proposition}\label{p-5.5} 
The following assertions hold for  $\tinf$ and $\widecheck{T}$.
\begin{itemize}
\item[(i)] $T(\linf)$ is a proper, weak-$\ast$ dense linear subspace of 
$\textnormal{Ker}(\varphi_{1/w})$.
\item[(ii)]  $\widecheck{T}(\linf)\not\subseteq \linf$.
\item[(iii)]  $(I_{\linf}-Q)(\linf)\not\subseteq T(\linf)\not\subseteq \linf$, with $Q$ 
as defined in \eqref{eq-5.1}.
\item[(iv)] $\linf\not\subseteq \widecheck{T}(\linf)$. 
\item[(v)] $\linf\not\subseteq \widecheck{T}(\lexp)\not\subseteq \lexp$.
\item[(vi)]   $\lexp\not\subseteq \widecheck{T}(\lexp)$.
\end{itemize}
\end{proposition}

\begin{proof}
(i) It was shown in Remark \ref{r-5.3} that $T(\linf)\subsetneqq \textnormal{Ker}(\varphi_{1/w})$.
Recall,  if $X$ and $Y$ are Banach spaces and $S\colon X\to Y$ is a bounded linear operator
with adjoint operator $S^*\colon Y^*\to X^*$,
then the weak-$\ast$ closure of $S^*(Y^*)$ equals $\textnormal{Ker}(S)^\perp$, 
\cite[Theorems 4.7 and 4.12]{R}.
We apply this when $S$ is the operator  $\tlog$, 
in which case $(\tlog)^*=-\tinf$ with $\tinf\colon\linf\to\lexp$.
Accordingly, the weak-$\ast$ closure of $\tinf(\linf)$ is equal to $\textnormal{Ker}(\tlog)^\perp$.
Since $\textnormal{Ker}(\tlog)=\text{span}(1/w)$, 
\cite[Theorem 3.4]{COR-asnsp}, it follows that 
the weak-$\ast$ closure of $\tinf(\linf)$  equals $\textnormal{Ker}(\varphi_{1/w})$.

(ii) Direct computation shows, for $x\in(-1,1)$, that  
$$
\widecheck{T}(\chi_{[0,1)})(x)=
-w(x)T\Big(\frac{\chi_{[0,1)}}{w}\Big)(x)=\frac{-1}{\pi} \log\left|
\frac{(1-x)^{1/2}(1+x)^{-1/2}+1}{(1-x)^{1/2}(1+x)^{-1/2}-1}\right|;
$$
see the proof of \cite[Lemma 4.3]{OE}. This function is plainly unbounded on $(-1,1)$
whereas $\chi_{[0,1)}\in L^\infty$.

(iii) That $T(\linf)\not\subseteq \linf$ is clear as
$T(\chi_{(-1,1)})(x)=(1/\pi)\log\big((1-x)/(1+x)\big)$, for $x\in(-1,1)$.

Assume, by way of contradiction, that 
$(I_{\linf}-Q)(\linf)\subseteq T(\linf)$. We  show that
this implies $\widecheck{T}(\linf)\subseteq \linf$. So, fix $g\in\linf$.
Then $(I_{\linf}-Q)(g)=T(f)$, for some $f\in\linf$ and  so
$\widecheck{T}(g)=\widecheck{T}(I_{\linf}-Q)(g)+\widecheck{T}(Q(g))$.
But, $T(1/w)=0$ implies that 
$$
\widecheck{T}(Q(g))= \left(\frac{-1}{\pi}\int_{-1}^1\frac{g}{w}\right)
wT\Big(\frac1w\Big)=0.
$$
Hence, $\widecheck{T}(g)=\widecheck{T}(I_{\linf}-Q)(g)=\widecheck{T}T(f)$
which, by Theorem \ref{t-5.1}(ii), equals $f\in \linf$.
That is, $\widecheck{T}(\linf)\subseteq \linf$, which contradicts (ii).

(iv) Assume, by way of contradiction, that  $\linf\subseteq \widecheck{T}(\linf)$.
Then $T(\linf)\subseteq T\widecheck{T}(\linf)
=(I_{\linf}-Q)(\linf)\subseteq\linf$ by Theorem \ref{t-5.1}(iii).
Hence, $T(\linf)\subseteq \linf$, which is not true by part (iii).

(v) Assuming the contrary, that is, $\widecheck{T}(\lexp)\subseteq \lexp$ we will arrive
at the containment  $\widehat T(\logl)\subseteq\logl$, which does
not hold as shown in  \cite[Proposition 4.12(iv)]{COR-asnsp}.

Let $f\in\logl$ and select any $g\in\lexp$. Since $1/w\in L^{2,\infty}\subseteq L^{3/2}$
and $g\in\lexp\subseteq L^6$, it follows that $g/w\in L^{6/5}\subseteq \logl$,
as $1/(3/2)+1/6=1/(6/5)$. 
By assumption, for each $A\in\mathcal{B}$, the function  $\widecheck T(g\chi_A)\in\lexp$ and 
so H\"older's inequality yields that
$$
(wf)T\Big(\frac{g}{w}\chi_A\Big)=-f\widecheck{T}(g\chi_A)\in L^1.
$$ 
Since $A\in\mathcal{B}$ is arbitrary with  $wf\in L^1$ and $g/w\in  \logl$, 
Proposition \ref{p-3.1} implies
that $(g/w)T(wf)\in L^1$ and hence, that
$$
g\widehat T(f)=-\frac{g}{w}T(wf)\in L^1.
$$
Accordingly, $g\widehat T(f)\in L^1$ for every $g\in\lexp$ and so, 
$\widehat T(f)\in(\lexp)'=\logl$. Hence,
$\widehat T(\logl)\subseteq\logl$ which was noted to be false.

(vi)   
Assume that $\lexp\subseteq \widecheck{T}(\lexp)$, in which case 
$T(\lexp)\subseteq T\widecheck{T}(\lexp)$. Fix $2<p<\infty$.
As indicated in \eqref{eq-5.1-2}, the identity 
$T\widecheck{T}(h)=h-Q(h)$ holds in  $L^p$ with each of the individual terms
$h, Q(h)$ and $T\widecheck{T}(h)$ belonging to $L^p$. If $g\in\lexp$ then 
$g\in L^p$ and so $T\widecheck{T}(g)=g-Q(g)$. From the definition of
$Q$ it is clear that  $Q(g)\in L^\infty\subseteq\lexp$, and so
$g-Q(g)\in \lexp$. Hence, $T\widecheck{T}(g)\in \lexp$.
It follows that $T(\lexp)\subseteq\lexp$, that is, $T\colon\lexp\to\lexp$
boundedly (cf.\ Lemma 4.1). This is impossible as
$\lexp$ has trivial Boyd indices.
\end{proof}


Recall from \eqref{airfoil} that the airfoil equation is $T(f)=g$, that is, 
\begin{equation*}
\frac{1}{\pi}\, \mathrm{ p.v.}  \int_{-1}^{1}\frac{f(x)}{x-t}\,dx=g(t),
\quad \mathrm{a.e. }\; t\in(-1,1).
\end{equation*}
Given $g$, an  \textit{inversion formula} is needed to solve this equation for $f$. 
Theorem \ref{t-5.1} and Corollary \ref{c-5.2}  allow us to  obtain such an inversion formula  for 
$\tinf\colon\linf\to \lexp$.


\begin{theorem}\label{t-5.6}
Let $g$ belong to the range of  $\tinf\colon\linf\to \lexp$.  
The unique solution  $f\in \linf$ of the airfoil equation \eqref{airfoil} is 
$$
f=\widecheck{T}(g)=-wT\Big(\frac{g}{w}\Big).
$$
\end{theorem}

\begin{proof}
Define $f:=\widecheck T(g)$ and note that $f\in L^\infty$ by Corollary \ref{c-5.2}.
Remark \ref{r-5.2} implies that
$T(f)=T(\widecheck T(g))=g-Q_{\lexp}(g).$
But, $g=\tinf(h)$ for some $h\in L^\infty$ and so Theorem \ref{t-5.1}(iv) implies that
$$
Q_{\lexp}(g)= \left(\frac{1}{\pi}\int_{-1}^1\frac{T(h)}{w}\right)\chi_{(-1,1)}=0.
$$
Hence, $T(f)=g$. That is, $f=\widecheck T(g)$ is the unique solution 
(cf. Proposition \ref{p-3.4}) of the airfoil equation.
\end{proof}


According to Proposition \ref{p-5.5}(i) we can interpret the operator 
$\tinf$ as taking its values in the proper, closed subspace 
$\textnormal{Ker}(\varphi_{1/w})$ of $\lexp$.
Theorem \ref{t-5.6} implies that $T\colon L^\infty\to \textnormal{Ker}(\varphi_{1/w})$
is a vector space isomorphism onto its range 
$T(L^\infty)\subsetneqq \textnormal{Ker}(\varphi_{1/w})$
whose inverse map is the vector space isomorphism $\widecheck T\colon 
T(L^\infty)\to L^\infty$.

\begin{corollary}\label{c-5.7}
The linear map $\widecheck T\colon T(L^\infty)\to L^\infty$ is a closed
operator defined on the proper linear subspace $T(L^\infty)$ of $\textnormal{Ker}(\varphi_{1/w})$.
\end{corollary}

\begin{proof}
Consider any sequence $(g_n)_{n=1}^\infty\subseteq T(L^\infty)$ satisfying $g_n\to g$ in 
$\textnormal{Ker}(\varphi_{1/w})$ and $\widecheck T(g_n)\to f$ in $L^\infty$. 
Let $(h_n)_{n=1}^\infty\subseteq L^\infty$ satisfy $g_n=T(h_n)$, for each
$n\in\N$, in which case $T(h_n)\to g$ in $\textnormal{Ker}(\varphi_{1/w})$.
On the other hand, $g_n=T(h_n)$ implies, via Theorem \ref{t-5.6}, that
$h_n=\widecheck T(g_n)$, for each $n\in\N$, and so $h_n\to f$ in $L^\infty$.
By the continuity of $\tinf\colon L^\infty\to \textnormal{Ker}(\varphi_{1/w})$ 
it follows that $T(h_n)\to T(f)$ in $\textnormal{Ker}(\varphi_{1/w})$
and hence, that $g=T(f)$ with $f\in L^\infty$. That is, $g\in T(L^\infty)$. By the uniqueness
statement in Theorem \ref{t-5.6} we can conclude that $f=\widecheck T(g)$.
Accordingly, $\widecheck T\colon T(L^\infty)\to L^\infty$ is a closed
linear operator defined in $\textnormal{Ker}(\varphi_{1/w})$.
\end{proof}

\begin{remark}\label{r-5.9}
Observe that $\textnormal{Ker}(\varphi_{1/w})$ 
is a \textit{proper}, closed subspace of $\lexp$ 
which contains $T(L^\infty)$. This does not contradict Proposition
\ref{p-3.5} as $\textnormal{Ker}(\varphi_{1/w})$ is not an order ideal in $L^0$ and hence, is not a B.f.s.
Indeed, $f=\chi_{[0,1)}-\chi_{(-1,0)}$ is an odd function
and so $f\in \textnormal{Ker}(\varphi_{1/w})$. However, $|f|=\chi_{(-1,1)}$ satisfies 
$\varphi_{1/w}(|f|)=\pi$ and so $|f|\not\in\textnormal{Ker}(\varphi_{1/w})$.
\end{remark}


\section{The optimal domain for $T$ taking values in $\lexp$}
\label{S6}


In this section we  prove that  the operator $\tinf\colon\linf \to \lexp$  is optimally defined. 
The approach used for this is  first to explicitly determine
the  \textit{optimal domain} of $T$ with values in $\lexp$, denoted by $[T,\lexp]$,
and then to show that it is isomorphic to $\linf$. This is the same approach  that was used
to show that $T\colon X\to X$ is optimally defined whenever $X$
is a r.i.\ space with non-trivial Boyd indices (cf.\ \cite[\S3]{COR-mh}) and 
also used to show that 
$T\colon\logl \to L^1$ is optimally defined,  \cite[\S5]{COR-asnsp}.
This strategy required an in-depth study of the
optimal domain (cf.\ \cite[Theorem 5.3]{COR-ampa}, \cite[Theorem]{COR-mh})
that we will also use here.

Define the  space of functions
\begin{equation}\label{eq-6.1}
[T, \lexp] : = \Big\{f \in L^1: T(h) \in \lexp \text{ for all } |h|  \le |f|\Big\},
\end{equation}
referred to  as the  \textit{optimal domain} of $T$ with values in $\lexp$.
Proposition \ref{p-3.4} implies that the inclusion $\linf \subseteq [T, \lexp] $ is valid.
It will be shown that the opposite containment  $[T, \lexp]\subseteq  \linf $ also  holds.


We begin by investigating the structure and exhibiting various properties of $[T, \lexp]$.

\begin{proposition}\label{p-6.1} 
\mbox{}
\begin{itemize}
\item[(i)]
Let  $f \in L^1$.  The following conditions are equivalent.
\begin{itemize}
\item[(a)]  $f\in[T,\lexp]$.
\item[(b)] $\displaystyle \sup_{|h|\le|f|}\|T(h)\|_{\lexp}<\infty.$
\item[(c)] $T(f\chi_A)\in \lexp$ for every $A\in\mathcal{B}$.
\item[(d)] $\displaystyle \sup_{A\in\mathcal{B}}\|T(f\chi_A)\|_{\lexp}<\infty.$
\item[(e)] $T(\theta f)\in \lexp$ for every $\theta\in L^\infty$ with $|\theta|=1$ a.e.
\item[(f)] $\displaystyle \sup_{|\theta|=1}\|T(\theta f)\|_{\lexp}<\infty.$
\item[(g)] $fT(g) \in L^1$ for every $ g \in \logl$. 
\end{itemize}
\item[(ii)]
Every function $f\in[T,\lexp]$ satisfies the inequalities
\begin{equation*}
\sup_{A\in\mathcal{B}}\big\|T(\chi_A f)\big\|_{\lexp}
\le
\sup_{|\theta|=1}\big\|T(\theta f)\big\|_{\lexp}
\le
\sup_{|h|\le|f|}\big\|T(h)\big\|_{\lexp}
\le
4 \sup_{A\in\mathcal{B}}\big\|T(\chi_A f)\big\|_{\lexp}.
\end{equation*}
\end{itemize}
\end{proposition}

\begin{proof}
The proof of the equivalences (a)$\iff$(b)$\iff$(c)$\iff$(d)$\iff$(e)$\iff$(f) in part (i) 
follows by adapting the proof of Proposition 4.1 in \cite{COR-ampa}
to the case when $X=\lexp$,  after observing that the non-triviality of the Boyd indices
of $X$ is not used in these proofs, and by using Proposition \ref{p-3.1} above
in place of \cite[Proposition 3.1]{COR-ampa}.

Applying the special case of Proposition \ref{p-3.1} (i.e., when $f\in \linf$ and
$g\in\logl$) it can be argued as  in the proof of  (d)$\Rightarrow$(g) in
Proposition 4.1 of \cite{COR-ampa} to conclude that  (b)$\Rightarrow$(g).

Finally, suppose that (g) is satisfied. Fix any $A\in\mcB$. Then $(f\chi_A)T(g)\in L^1$
for every $g\in\logl$ (by assumption). Apply Proposition \ref{p-3.1} to
$f\chi_A$ in place of $f$ to obtain that $gT(f\chi_A)\in L^1$ for all $g\in\logl$. Accordingly,
$T(f\chi_A)\in (\logl)'=\lexp$, which establishes (c).

The inequalities in part (ii) can be proved  exactly as  in the proof of 
part (ii) of Proposition 4.1 in \cite{COR-ampa}.
\end{proof}


\begin{proposition}\label{p-6.2}
The space $[T, \lexp] $ is a linear lattice for the a.e.\ pointwise order and,
when equipped with the norm
\begin{equation}\label{eq-6.2}
\|f\|_{[T, \lexp]}:= \sup\Big\{\|T(h)\|_{\lexp}: |h|\le|f|
\Big\},\quad f \in [T, \lexp] ,
\end{equation}
it is a B.f.s.\ over $(-1,1)$.

Moreover,  $[T, \lexp] $ is the largest B.f.s., within $L^0$,   
to which   $\tinf\colon\linf \to \lexp$ admits an $\lexp$-valued, 
continuous linear extension.
\end{proposition}

\begin{proof}
Proposition \ref{p-6.1} shows that $\|\cdot\|_{[T, \lexp]}$ is finite
on $[T,\lexp]$. That $[T,\lexp]$ is a vector space and  $\|\cdot\|_{[T, \lexp]}$ is a lattice
seminorm follows the argument given on p.1850 of \cite{COR-ampa}. To show that
$\|\cdot\|_{[T, \lexp]}$ is actually a norm, let $f\in[T,\lexp]$ satisfy
$\|f\|_{[T,\lexp]}=0$. Suppose that $f\not=0$. Choose 
$A\in\mathcal{B}$ with $\mu(A)>0$ such that $f\chi_A\in\linf$ and 
$f(x)\not=0$ for all $x\in A$. Since $|f\chi_A|\le |f|$, it follows from \eqref{eq-6.2} that
$\|T(f\chi_A)\|_{\lexp}=0$ and so $T(f\chi_A)=0$. The injectivity of
$\tinf\colon\linf\to\lexp$ implies that $f\chi_A=0$ a.e.
This contradiction establishes the claim.
Completeness of $[T,\lexp]$  follows the arguments given in the proof of
Lemma 4.4 in \cite{COR-ampa}. The proof that $[T,\lexp]$ satisfies the Fatou property
can be deduced from  the arguments used to establish Proposition 4.5 in \cite{COR-ampa}.

Establishing the optimality of $[T, \lexp]$ follows the proof of Theorem 4.6 of 
\cite{COR-mathnach} as explicitly done in the proof of Proposition 2.7 in \cite{COR-qm}.

We point out that  the proofs from \cite{COR-ampa} indicated 
above correspond to the case of $[T,X]$ when
$X$ is a r.i.\ space with non-trivial Boyd indices. However, they
remain valid in our current setting since the non-triviality of the Boyd indices
is not actually used in those proofs from \cite{COR-ampa}.
\end{proof}


The following remarkable inequality, namely Theorem 3.11 of \cite{COR-mathnach}, 
will be needed in the sequel.

\begin{theorem}\label{t-6.3}
Let $A\in\mathcal{B}$ satisfy $\mu(A)>0$. Then
$$
\big\|T(\chi_{A})\big\|_{\lexp}> \frac{1}{\pi e^2}.
$$
\end{theorem}


\begin{proposition}\label{p-6.4}
The inclusion $[T,\lexp]\subseteq\linf$ is valid.
\end{proposition}

\begin{proof}
Let $f\in[T,\lexp]$. For each $n\in\N$, set $A_n:=\{x\in(-1,1): n\le|f(x)|<n+1\}$.
Note that $n\chi_{A_n}\le|f|$ a.e. Let $g\in\logl$. By applying 
Proposition \ref{p-3.1} to $n\chi_{A_n}\in\linf$ and $g\in\logl$
we can deduce
that both $gT(n\chi_{A_n})$ and $n\chi_{A_n}T(g)$ are integrable and satisfy
$$
\int_{-1}^1gT(n\chi_{A_n})=-\int_{-1}^1n\chi_{A_n}T(g),\quad n\in\N.
$$
Thus,  using the equivalence (a)$\iff$(g) in Proposition \ref{p-6.1}   
and the inequality $n\chi_{A_n}\le|f|$ yields
$$
\left|\int_{-1}^1gT(n\chi_{A_n})\right|\le \int_{-1}^1|fT(g)|<\infty,\quad n\in\N.
$$
Since the previous inequality holds for arbitrary $g\in\logl$,  it follows that the
set $\{n T(\chi_{A_n}):n\in\N\}$ is weak-$\ast$ bounded in $\lexp$ and hence, 
by the uniform boundedness principle, it is norm bounded in
$\lexp=(\logl)^*$. That is, for some constant $C>0$, we have
$$
\big\|T(\chi_{A_n})\big\|_{\lexp}\le \frac{C}{n},\quad n\in\N,
$$
which implies that $\|T(\chi_{A_n})\|_{\lexp}\to0$ for $n\to\infty$. 
But, Theorem \ref{t-6.3} implies that
$$
\big\|T(\chi_{A})\big\|_{\lexp}> \frac{1}{\pi e^2},
$$
for every $A\in\mathcal{B}$ with $\mu(A)>0$.
Consequently, there exists $n_0\in\N$ such that $\mu(A_n)=0$ for all $n\ge n_0$,
that is, $|f(x)|\le n_0$ for a.e.\ $x\in(-1,1)$. So, $f\in \linf$.
\end{proof}


We  now answer the question regarding the optimal extension of 
$\tinf\colon\linf \to \lexp$.

\begin{theorem}\label{t-6.5}
The identity $[T, \lexp] = \linf$ holds as an order 
and bicontinuous isomorphism between B.f.s.'  Hence, 
$\tinf\colon\linf \to \lexp$ does not admit a continuous, linear,  $\lexp$-valued 
extension to any strictly  larger B.f.s.\ within $ L^0$.
\end{theorem}

\begin{proof}
Propositions \ref{p-6.2} and  \ref{p-6.4} show that $\linf$ and
$[T,  \lexp]$ are equal as vector spaces.

To prove that $\linf$ is continuously included in $[T,  \lexp] $,
fix $f\in\linf$. For each $h\in L^0$ satisfying $|h|\le|f|$ observe that
$h\in\linf$ with $\|h\|_{\linf}\le\|f\|_{\linf}$. Proposition \ref{p-3.2} yields that 
$\|T(h)\|_{\lexp}\le \|T\|\cdot\|h\|_{\linf}\le \|T\|\cdot\|f\|_{\infty}$ and so we can
conclude from \eqref{eq-6.2} that 
$$
\|f\|_{[T, \lexp]}\le \|T\|\cdot\|f\|_{\linf},\quad f\in\linf,
$$
where $\|T\|$ denotes the operator norm $\|T\|_{\linf\to  \lexp}$.
That is, the natural inclusion $\linf\subseteq[T,  \lexp]$ is continuous.

Since both $\linf$ and $[T, \lexp]$ are B.f.s.' (cf.\ Proposition \ref{p-6.2}) it follows
from the Open Mapping Theorem 
that the natural inclusion $\linf\subseteq[T,  \lexp]$ is an isomorphism.

Proposition \ref{p-6.2} implies that $\tinf\colon\linf \to\lexp$ 
does not admit a continuous, linear,  $\lexp$-valued
extension to any strictly  larger B.f.s.\  within $ L^0$.
\end{proof}


\subsection*{Acknowledgements}
The first author acknowledges the support  of 
PID2024-155593NB-C21 (FEDER(EU) / Ministerio de Ciencia e Innovaci\'on-Agencia 
Estatal de Investigaci\'on),
IMUS-Mar\'{\i}a de Maeztu grant CEX2024-001517-M 
and FQM-262 (Junta de Andaluc\'{\i}­a).




\begin{thebibliography}{99}

\bibitem{APS}
R. Alaifari, L. B. Pierce, S. Steinerberger
\textit{Lower bounds for the truncated Hilbert transform},
Rev. Mat. Iberoam. \textbf{32 }(2016),  23--56.

\bibitem{BR}
C. Bennett,   K. Rudnick, \textit{On Lorentz--Zygmund spaces},
Dissertationes Math.  \textbf{175} (1980), 1--67.

\bibitem{BS} 
C. Bennett, R. Sharpley, 
\textit{Interpolation of Operators}, 
Academic Press,  Boston, 1988.

\bibitem{bertola-etal}
M. Bertola, A. Katsevich, A. Tovbis,
\textit{Singular value decomposition of a finite Hilbert transform defined 
on several intervals and the interior problem of tomography: 
the Riemann-Hilbert problem approach},
Comm. Pure Appl. Math.   \textbf{69} (2016), 407--477. 

\bibitem{cheng-rott}
H. K. Cheng, N. Rott, 
\textit{Generalizations of the inversion formula of thin airfoil theory},
J. Rational Mech. Anal. \textbf{3} (1954), 357--382. 

\bibitem{COR-ampa}
G. P. Curbera, S. Okada, W. J. Ricker, 
\textit{Inversion and extension of the finite Hilbert transform on $(-1,1)$},
Ann. Mat. Pura  Appl. (4) \textbf{198} (2019), 1835--1860.

\bibitem{COR-qm} 
G. P. Curbera, S. Okada, W. J. Ricker, 
\textit{Extension and integral representation of the finite 
Hilbert transform in rearrangement invariant spaces},
Quaest. Math. \textbf{43} (2020), 783--812.

\bibitem{COR-mh}
G. P. Curbera, S. Okada, W. J. Ricker, 
\textit{Non-extendability of the finite Hilbert transform},
Monatsh.  Math. (4) \textbf{195} (2021), 649--657.

\bibitem{COR-am}
G. P. Curbera, S. Okada, W. J. Ricker,
\textit{Fine spectra of the finite Hilbert transform in function spaces},
Adv.  Math. \textbf{380} (2021), 107597.

\bibitem{COR-mathnach}
G. P. Curbera, S. Okada, W. J. Ricker,
\textit{Measure theoretic aspects of the finite Hilbert transform},
Math. Nachr. \textbf{297} (2024), 3927--3942.

\bibitem{COR-asnsp}
G. P. Curbera, S. Okada, W. J. Ricker,
\textit{The finite   Hilbert transform acting in the Zygmund space  
$\logl$}, Ann. Sc. Norm. Super. Pisa Cl. Sci. (5) \textbf{25} (2024), 1527--1557.

\bibitem{duduchava} 
R. Duduchava,
\textit{Integral Equations in Convolution with Discontinuous Presymbols. Singular Integral 
Equations with Fixed Singularities, and their Applications to some Problems of Mechanics},
Teubner Verlagsgesellschaft, Leipzig, 1979.
 
\bibitem{gakhov} 
F. D. Gakhov,   \textit{Boundary Value Problems}, 
Dover Publications, Inc., New York, NY, 1990.

\bibitem{GG}
I. M. Gel'fand, M. I. Graev, 
\textit{Crofton function and inversion formulas
in real integral geometry}, Funct. Anal. Appl. \textbf{25}  (1991), 1--5.

\bibitem{gohberg-krupnik-1} 
I. Gohberg, N. Krupnik, 
\textit{One-Dimensional Linear Singular Integral Operators Vol. I. Introduction}, 
Operator Theory Advances and Applications \textbf{53}, Birkh\"auser,  Berlin, 1992.

\bibitem{gohberg-krupnik-2} 
I. Gohberg, N. Krupnik, 
\textit{One-Dimensional Linear Singular Integral Operators Vol. II. General Theory
and Applications}, 
Operator Theory Advances and Applications \textbf{54}, Birkh\"auser,  Berlin, 1992.

\bibitem{katsevich-tovbis}
A. Katsevich, A. Tovbis, \textit{Finite Hilbert transform with 
incomplete data: null-space and singular values},  
Inverse Problems \textbf{28} (2012), 105006, 28 pp.

\bibitem{king}  
F. W. King, \textit{Hilbert Transforms
Vol. I}, Cambridge University Press, Cambridge New York, 2009.

\bibitem{KPS}
S. G. Krein,   Ju. I. Petunin,  E. M. Semenov, 
\textit{Interpolation of Linear Operators}, 
Amer. Math. Soc., Providence, R. I., 1982.

\bibitem{LT}
J. Lindenstrauss,  L. Tzafriri, \textit{Classical Banach Spaces
Vol. II}, Springer-Ver\-lag, Berlin, 1979.

\bibitem{Lu}
W. A. J. Luxemburg, \textit{Banach Function Spaces}, 
PhD Thesis, Delft Institute of Technology, The Netherlands, 1955.


\bibitem{ME}
W. McLean, D. Elliott, 
\textit{On the p-norm of the truncated Hilbert transform},
Bull. Austral. Math. Soc.  \textbf{38} (1988),  413--420.

\bibitem{mikhlin-prossdorf}  
S. Mikhlin, S. Pr\"ossdorf, 
\textit{Singular Integral Operators}, Springer-Verlag, Berlin, 1986.

\bibitem{muskhelishvili}  
N. I. Muskhelishvili, 
\textit{Singular Integral Equations}, 
Wolters-Noordhoff Publishing, Groningen, 1967

\bibitem{OE}
S. Okada, D. Elliott, \textit{The finite Hilbert transform in $\mathcal{L}^2$},
Math. Nachr. \textbf{153} (1991),  43--56.

\bibitem{okada-ricker-sanchez}
S. Okada, W. J. Ricker, E. A. S\'anchez-P\'erez, 
\textit{Optimal Domain and Integral Extension of Operators: Acting in Function Spaces}, 
Operator Theory Advances and Applications \textbf{180}, Birkh\"auser,  Berlin, 2008.

\bibitem{P}
S. K. Pichorides, 
\textit{On the best values of the constants in the theorems 
of M. Riesz, Zygmund and Kolmogorov},
Studia Math. \textbf{44} (1972), 165--179.

\bibitem{reissner}
E. Reissner, \textit{Boundary value problems in 
aerodynamics of lifting surfaces in non-uniform motion}, 
Bull. Amer. Math. Soc. \textbf{55} (1949), 825--850. 

\bibitem{R}
W. Rudin, 
\textit{Functional Analysis},
McGraw Hill, New York-St. Louis-San Francisco, 1973.

\bibitem{sidky-etal}
E. Y. Sidky, Xiaochuan Pan, \textit{Recovering a compactly 
supported function from knowledge of its Hilbert transform on a finite interval},
IEEE Signal Processing Letters \textbf{12} (2005), 97--100.

\bibitem{sohngen}  
H. S\"ohngen, \textit{Zur Theorie der endlichen Hilbert-Transformation}, 
Math. Z. \textbf{60} (1954),  31--51.

\bibitem{tricomi-1}
F. G. Tricomi,   \textit{On the finite Hilbert transform}, 
Quart. J. Math. \textbf{2} (1951),  199--211.

\bibitem{tricomi}
F. G. Tricomi, \textit{Integral Equations}, 
Interscience, New York, 1957.

\bibitem{Z}
A. Zygmund, 
\textit{Trigonometric Series},
Cambridge University Press, Cambridge, 1959.

\end{thebibliography}
\end{document}